\documentclass[11pt, reqno]{amsart}

\usepackage{amsfonts}
\usepackage{amssymb}
\usepackage{amsthm}
\usepackage{amsmath}
\usepackage{a4wide}
\usepackage{mathrsfs}
\usepackage{epsfig}
\usepackage{palatino}
\usepackage{tikz,fp,ifthen}
\usepackage{float}
\usepackage{graphicx,cite}

\usepackage{graphicx}

\usepackage{color}
\definecolor{citegreen}{rgb}{0,0.6,0}
\definecolor{refred}{rgb}{0.8,0,0}
\usepackage[colorlinks, citecolor=citegreen,linkcolor=refred]{hyperref}

\usepackage{mathtools}
\mathtoolsset{showonlyrefs}

\theoremstyle{plain}
\newtheorem{thm}{Theorem}[section]

\newtheorem{prop}[thm]{Proposition}
\newtheorem{cor}[thm]{Corollary}
\newtheorem{ackn}{Acknowledgement\!}
\newtheorem{thmm}{Theorem\!}

\theoremstyle{definition}
\newtheorem{defn}[thm]{Definition}
\newtheorem{conge}[thm]{Conjecture}

\theoremstyle{remark}
\newtheorem{rem}[thm]{Remark}

\numberwithin{equation}{section}

\def\Ri{\R^{n+1}}

\def\eps{\varepsilon}

\def\Ric{{\mathrm {Ric}}}

\def\R{\mathbb R}

\def\R{{{\mathbb R}}}

\def\NN{\mathbb N}

\newcommand{\intbar}{\etaathop{\int\etaakebox(-13.5,0){\rule[4pt]{.7em}{0.3pt}}
\kern-6pt}\nolimits}

\newcommand{\be}{\begin{equation}}
\newcommand{\ee}{\end{equation}}
\newcommand{\bea}{\begin{equation*}}
\newcommand{\eea}{\end{equation*}}

\newcommand{\Hess}{\mathrm{Hess}}

\usetikzlibrary{backgrounds}
\usetikzlibrary{decorations.pathmorphing,backgrounds,fit,calc,through}
\usetikzlibrary{arrows}
\usetikzlibrary{shapes,decorations,shadows}
\usetikzlibrary{fadings}
\usetikzlibrary{patterns}
\usetikzlibrary{mindmap}
\usetikzlibrary{decorations.text}
\usetikzlibrary{decorations.shapes}

\begin{document}

\title{Semilinear Li~\&~Yau inequalities}

\author{Daniele Castorina}
\address{Daniele Castorina\\
	Dipartimento di Matematica e Applicazioni,
	Universit\`a di Napoli, Via Cintia, Monte S. Angelo 80126 Napoli,
	Italy}
\email{daniele.castorina@unina.it}

\author{Giovanni Catino}
\address{Giovanni Catino\\
	Dipartimento di Matematica,
	Politecnico di Milano, Piazza Leonardo da Vinci 32, 20133, Milano,
	Italy}
\email{giovanni.catino@polimi.it}

\author{Carlo Mantegazza}
\address{Carlo Mantegazza\\
	Dipartimento di Matematica e Applicazioni,
	Universit\`a di Napoli, Via Cintia, Monte S. Angelo 80126 Napoli,
	Italy}
\email{c.mantegazza@sns.it}

\begin{abstract} We derive an adaptation of Li~\&~Yau estimates for positive solutions of semilinear heat equations on Riemannian manifolds with nonnegative Ricci tensor. We then apply these estimates to obtain a Harnack inequality and to discuss monotonicity, convexity, decay estimates and triviality of ancient and eternal solutions. 
\end{abstract}

\subjclass[2010]{35K05, 58J35}

\maketitle

\section{Introduction}

We show some Li~\&~Yau estimates (see~\cite{liyau}) for positive solutions of semilinear heat equations $u_t = \Delta u + u^p$ in $\R^n$, or in a complete Riemannian manifold $(M,g)$ with nonnegative Ricci tensor, where $p>0$ and we derive an associated Harnack inequality. Similar results can be found in the papers~\cite{jiayuli,wuyang}.

We will then discuss the application of these inequalities and methods to the analysis of positive classical {\em ancient} and {\em eternal} solutions.
\noindent We call a solution
\begin{itemize}
\item {\em ancient} if it is defined in $M\times(-\infty,T)$ for some $T\in\R$,

\item {\em immortal} if it is defined in $M\times(T,+\infty)$ for some $T\in\R$,

\item {\em eternal} if it is defined in $M\times\R$.
\end{itemize}
Furthermore, we call a solution $u$ {\em static} if is independent of time (\textit{i.e. }$u(x,t)=u(x)$, hence it satisfies $\Delta u + u^p=0$); {\em trivial} if it is constant in space (that is, $u(x,t)=u(t)$ and solves the ODE $u'=u^p$). We say that $u$ is simply {\em constant} if it is constant in space and time.\\
Notice that positive ancient (trivial) solutions always exist (by solving the corresponding ODE), while being eternal is quite restrictive: for instance, by an easy maximum principle argument, if $M$ is compact and $u$ is eternal then $u\equiv0$ (see~\cite[Corollary~2.2]{cama2}, for instance). In the noncompact situation, this is not true: the following ``Talenti's  function''~\cite{tal}
$$
u(x)=\frac{24}{\bigl(1+|x|^2\bigr)^2}
$$
on $\R^6$ satisfies $\Delta u + u^2 = 0$, in particular $u$ is a nonzero eternal solution of $u_t=\Delta u + u^2$.

A reason for the interest in ancient or eternal solutions is that they typically arise as blow--up limits when the solutions of semilinear parabolic equations (in bounded intervals) develop a singularity at a certain time (i.e. the solution becomes unbounded). They also appear naturally and play a key role in the analysis of mean curvature flow and of Ricci flow (from which we get several suggestions), which are also described by (much more complicated systems of) parabolic PDEs. In such cases, the solutions are respectively, evolving hypersurfaces and abstract Riemannian manifolds. Analyzing their properties and eventually classifying them lead to understand the behavior of the solutions close to the singularity or even (in very lucky situations, notably the motion by curvature of embedded curves in the plane and the 2--dimensional Ricci flow) to actually {\em exclude} the formation (existence) of singularities. Moreover, this analysis can also be used to get uniform (universal) estimates on the ``rate'' a solution (or some related quantity) becomes unbounded at a singularity. Indeed (roughly speaking), typically a ``faster'' rate implies that performing a blow--up at a singularity, we obtain a bounded, nonzero, nonconstant, eternal solution, while with the slower ``standard'' natural rate, we get an ancient solution (immortal solutions are usually less significant). This has been done, for example, by  Pol\'a\v cik, Quittner and Souplet through "universal estimates" for semilinear equations in $\R^n$ (see~\cite{pqs1,pqs2}) or by Hamilton through a technique of "smart point picking", which is more suitable for geometric flows (see~\cite{hamilton9}).\\
For instance, excluding the existence of bounded, positive, nonconstant, eternal solutions to the  equation $u_t = \Delta u + u^p$ in $\R^n\times\R$, we have a (universal, up to a constant) $L^\infty$ bound from above on every solution, approaching the singular time $T$.\\
We will see how from the Li~\&~Yau estimates (holding for the exponent $p$ in a suitable range), one can exclude the existence of such eternal solutions. Instead, unfortunately, even if we obtain some strong properties (pointwise monotonicity and convexity in time) of ancient solutions on Riemannian manifolds with nonnegative Ricci tensor (for suitable exponents $p$), we are still not able to show the ``natural'' conjecture that they are trivial, constant in space, as it is known when the ambient is $\R^n$, by a result of Merle and Zaag~\cite{merlezaag}.

There is a quite large literature on the subject, for a rather complete account we refer the interested reader to the paper of Souplet and Zhang~\cite{souzha} as well as the book of Quittner and Souplet ~\cite{qsbook}. Very important developments for the Euclidean case (some of them very recent) are obtained in~\cite{pqs2,quittner,quittner2}, while some extensions to the case of Riemannian manifolds can be found in~\cite{jiayuli,wuyang,zhugrad}.

We mention that all the conclusions also hold for ancient solutions of $u_t = \Delta u + |u|^p$, without assuming their positivity, by the following theorem.

\begin{thmm}[Theorem~2.6 in~\cite{cama2}]
Let $(M,g)$ be a complete Riemannian manifold with Ricci tensor  uniformly bounded below and let $u$ be an ancient solutions of $u_t = \Delta u + |u|^p$, then either $u\equiv 0$ or $u>0$ everywhere.
\end{thmm}

{\em In the whole paper the Riemannian manifolds $(M,g)$ will be smooth, complete, connected and without boundary. We will denote with $\Delta$ the associated Laplace--Beltrami operator. As it is standard, we will write $\Ric\geq0$ and we will say that the Ricci tensor is nonnegative with the meaning that all its eigenvalues are nonnegative. Finally, all the solutions we will consider are classical, $C^2$ in space and $C^1$ in time.}

\begin{ackn} We would like to thank Giacomo Ascione for his help with Mathematica\textsuperscript{TM}.
\end{ackn}

\section{Li~\&~Yau--type estimates for semilinear heat equations}\label{liyausem}

We now prove some Li~\&~Yau--type estimates for positive solutions of the semilinear heat equation $u_t = \Delta u + u^p$ with $p>1$, on a complete $n$--dimensional Riemannian manifold $(M,g)$ with nonnegative Ricci tensor $\Ric$.

For the sake of simplicity and more clarity we first assume that $M$ is compact.

Let $u:M\times[0,T)\to\R$ be a positive solution of $u_t = \Delta u + u^p$, then setting 
$f=\log u$, we have
$$
\vert\nabla f\vert=\frac{\vert\nabla u\vert}{u}\qquad\qquad \Delta f=\frac{\Delta u}{u}-\frac{\vert\nabla u\vert^2}{u^2}=\frac{\Delta u}{u}-\vert\nabla f\vert^2
$$
$$
f_t=\frac{u_t}{u}=\frac{\Delta u}{u}+u^{p-1}=\Delta f+\vert\nabla f\vert^2+e^{(p-1)f}.
$$
We consider $\alpha,\beta\in[0,1]$ and after defining
\begin{align*}
F=&\,t(f_t-\alpha \vert\nabla f\vert^2-\beta e^{f(p-1)})\\
=&\,t\bigl(\frac{u_t}{u}-\alpha\frac{\vert\nabla u\vert^2}{u^2}-\beta u^{p-1}\bigr)\\
=&\,t\bigl(\frac{\Delta u}{u}-\alpha\frac{\vert\nabla u\vert^2}{u^2}+(1-\beta)u^{p-1}\bigr)\\
=&\,t(\Delta \log u+(1-\alpha)\frac{\vert\nabla u\vert^2}{u^2}+(1-\beta)u^{p-1}\bigr)\\
=&\,t(\Delta f+(1-\alpha)\vert\nabla f\vert^2+(1-\beta)e^{(p-1)f})
\end{align*}
we compute, for $t>0$,
\begin{align*}
(\partial_t-\Delta)F=&\,F/t+t(\partial_t-\Delta)\Delta f+t(p-1)(1-\beta)[(\partial_t-\Delta)f]e^{(p-1)f}\\
&\,-t(p-1)^2(1-\beta)\vert \nabla f\vert^2 e^{(p-1)f}+2t(1-\alpha)\nabla f\nabla f_t -t(1-\alpha)\Delta\vert\nabla f\vert^2\\
=&\,F/t+t\bigl\{\Delta (\vert\nabla f\vert^2+e^{(p-1)f})
+(p-1)(1-\beta)(\vert\nabla f\vert^2+e^{(p-1)f})e^{(p-1)f}\\
&\,-(p-1)^2(1-\beta)\vert \nabla f\vert^2 e^{(p-1)f}+2(1-\alpha)\nabla f\nabla f_t -(1-\alpha)\Delta\vert\nabla f\vert^2\bigr\}\\
=&\,F/t+t\bigl\{(p-1)\Delta fe^{(p-1)f}+(p-1)^2\vert\nabla f\vert^2 e^{(p-1)f}+(p-1)(1-\beta)e^{2f(p-1)}\\
&\,+(p-1)(1-\beta)\vert\nabla f\vert^2e^{(p-1)f}-(p-1)^2(1-\beta)\vert \nabla f\vert^2 e^{(p-1)f}\\
&\,+2(1-\alpha)\nabla f\nabla[\Delta f+\vert\nabla f\vert^2+e^{(p-1)f}]+\alpha\Delta\vert\nabla f\vert^2\bigr\}\\
=&\,F/t+t\bigl\{(p-1)Fe^{(p-1)f}/t+(p-1)[\alpha+\beta(p-2)]\vert \nabla f\vert^2 e^{(p-1)f}\\
&\,+2(1-\alpha)\nabla f\nabla[\Delta f+\vert\nabla f\vert^2+e^{(p-1)f}]+2\alpha\vert\Hess\, f\vert^2+2\alpha\nabla f\Delta\nabla f\bigr\}.
\end{align*}
By the nonnegativity of the Ricci tensor there holds $\nabla f\Delta\nabla f\geq \nabla f\nabla \Delta f $ and using the inequality $\vert\Hess\, f\vert^2\geq(\Delta f)^2/n$, we have, as 
$$
\Delta f=F/t-(1-\alpha)\vert\nabla f\vert^2-(1-\beta)e^{(p-1)f},
$$
\begin{align}
(\partial_t-\Delta)F
\geq &\,F/t+t\bigl\{2\alpha(\Delta f)^2/n+(p-1)Fe^{(p-1)f}/t+(p-1)[\alpha+\beta(p-2)]\vert \nabla f\vert^2 e^{(p-1)f}\nonumber\\
&\,+2(1-\alpha)\nabla f\nabla[\Delta f+\vert\nabla f\vert^2+e^{(p-1)f}]+2\alpha\nabla f\nabla\Delta f\bigr\}\nonumber\\
=&\,F/t+t\bigl\{2\alpha(\Delta f)^2/n+(p-1)Fe^{(p-1)f}/t+(p-1)[\alpha+\beta(p-2)]\vert \nabla f\vert^2 e^{(p-1)f}\nonumber\\
&\,+2\nabla f\nabla[\Delta f+(1-\alpha)\vert\nabla f\vert^2+(1-\alpha)e^{(p-1)f}]\bigr\}\nonumber\\
=&\,F/t+t\bigl\{2\alpha(\Delta f)^2/n+(p-1)Fe^{(p-1)f}/t+2\nabla f\nabla F/t\nonumber\\
&\,+(p-1)(\beta p-\alpha)\vert \nabla f\vert^2 e^{(p-1)f}\bigr\}\nonumber\\
=&\,F/t+t\bigl\{2\alpha F^2/nt^2+2\alpha(1-\beta)^2e^{2f(p-1)}/n+[(p-1)-4\alpha(1-\beta)/n]Fe^{(p-1)f}/t\nonumber\\
&\,+2\alpha(1-\alpha)^2\vert\nabla f\vert^4/n-4\alpha(1-\alpha)\vert\nabla f\vert^2F/nt\\
&\,+2\nabla f\nabla F/t+[(p-1)(\beta p-\alpha)+4\alpha(1-\alpha)(1-\beta)/n]\vert \nabla f\vert^2 e^{(p-1)f}\bigr\}.\label{eq960}
\end{align}
Hence, at a point where $F\leq 0$ we conclude
\begin{align*}
(\partial_t-\Delta)F\geq &\,F/t+2\alpha F^2/nt+2t\alpha(1-\beta)^2e^{2f(p-1)}/n+[(p-1)-4\alpha(1-\beta)/n]Fe^{f(p-1)}\\
&\,+2\nabla f\nabla F+t[(p-1)(\beta p-\alpha)+4\alpha(1-\alpha)(1-\beta)/n]\vert \nabla f\vert^2 e^{(p-1)f}\\
\geq &\,F/t+2\alpha F^2/nt+2t\alpha(1-\beta)^2e^{2f(p-1)}/n+[(p-1)-4\alpha(1-\beta)/n]Fe^{(p-1)f}\\
&\,+2\nabla f\nabla F,
\end{align*}
provided that
\begin{equation}\label{cond2}
(p-1)(\beta p-\alpha)+4\alpha(1-\alpha)(1-\beta)/n\geq 0.
\end{equation}
Now we deal with the sum
$$
2t\alpha(1-\beta)^2w^2/n+[(p-1)-4\alpha(1-\beta)/n]Fw,
$$
where we set $w=e^{(p-1)f}\geq 0$, and we see that the minimum of this quadratic in $w$ is given by
$$
-F^2\frac{[(p-1)-4\alpha(1-\beta)/n]^2n}{8t\alpha(1-\beta)^2}=
-F^2\frac{[n(p-1)-4\alpha(1-\beta)]^2}{8nt\alpha(1-\beta)^2}
$$
Hence, after substituting in the above formula and using Hamilton's trick (see~\cite{hamilton2} or~\cite[Lemma~2.1.3]{Manlib}), being $M$ compact we get
\begin{align*}
\frac{d\,}{dt} F_{\min}(t)\geq
&\,\frac{F_{\min}(t)}{t}+\frac{F^2_{\min}(t)}{nt}\left[2\alpha-\frac{[n(p-1)-4\alpha(1-\beta)]^2}{8\alpha(1-\beta)^2}\right]\\
=&\,\,\frac{F_{\min}(t)}{t}+F^2_{\min}(t)\frac{n(p-1)}{8t\alpha(1-\beta)^2}\biggl[1+\frac{8\alpha(1-\beta)}{n}-p\biggr]\\
=&\,\,\frac{F_{\min}(t)}{t}+\frac{\varepsilon F^2_{\min}(t)}{t},
\end{align*}
at almost every $t>0$, where $F_{\min}\leq 0$ and we set
\begin{equation}\label{epseq}
\varepsilon=\varepsilon(n,p,\alpha,\beta)=\frac{n(p-1)}{8\alpha(1-\beta)^2}\biggl[1+\frac{8\alpha(1-\beta)}{n}-p\biggr].
\end{equation}
If now $\varepsilon$ is positive, that is,
\begin{equation}\label{cond1}
p< 1+\frac{8\alpha(1-\beta)}{n},
\end{equation}
when $F_{\min}\leq -1/\varepsilon$, we have $\frac{d\,}{dt}F_{\min}(t)>0$. Since $F_{\min}(0)=0$, this implies (``integrating'' this differential inequality) that
$F_{\min}\geq -1/\varepsilon$, hence
\begin{equation}
(f_t-\alpha \vert\nabla f\vert^2-\beta e^{(p-1)f})(x,t)=\frac{F(x,t)}{t}\geq F_{\min}(t)\geq -\frac{1}{\varepsilon t}\label{eq970}
\end{equation}
for every $x\in M$ and $t>0$. That is,
$$
\frac{u_t}{u}\geq \alpha \frac{\vert\nabla u\vert^2}{u^2}+\beta u^{p-1}-\frac{1}{\varepsilon t}
$$
or
\begin{equation}\label{eq980}
u_t\geq \alpha \frac{\vert\nabla u\vert^2}{u}+\beta u^p-\frac{u}{\varepsilon t}.
\end{equation}

Now the point is to find out for what exponents $p>1$ (depending on the dimension $n$) there exists constants $\alpha,\beta\in[0,1]$ satisfying conditions~\eqref{cond2} and~\eqref{cond1}, that is,
$$
(p-1)(\beta p-\alpha)+4\alpha(1-\alpha)(1-\beta)/n\geq 0
$$
$$
p< 1+\frac{8\alpha(1-\beta)}{n}.
$$
We find that these conditions are satisfied
\begin{itemize}
\item for $n\leq 3$, for every $p<8/n$ (choosing $\alpha=1$ and $\beta=1/p$\,) and for no $p\geq 8/n$,
\item for $n\geq 4$ (by using Mathematica\textsuperscript{TM}, see the Appendix), for every and only 
$$
p<\frac{3n+4+3\sqrt{n(n+4)}}{2(3n-4)}
$$
and some $\alpha,\beta\in(0,1)$ (notice the right side is larger than $8/n$\,).
\end{itemize}

\begin{defn}\label{pienne}
Given $p>1$, we define $\overline{p}_n=8/n$ for $n\leq 3$ and
$$
\overline{p}_n=\frac{3n+4+3\sqrt{n(n+4)}}{2(3n-4)}
$$
for $n\geq 4$.\\
We say that $\alpha\in(0,1]$ and $\beta\in[0,1)$ are ``admissible'' for $p\in(1,\overline{p}_n)$ if they satisfy the two conditions~\eqref{cond2} and~\eqref{cond1}, in such case we define $\varepsilon(n,p,\alpha,\beta)$ as in formula~\eqref{epseq}.
\end{defn}

\begin{rem}\label{betanonzero}
It must be $\alpha>0$ and $\beta<1$, otherwise condition~\eqref{cond1} would imply $p<1$. Moreover, notice that for every $p\in(1,\overline{p}_n)$ there is always an admissible pair $(\alpha,\beta)$ with $\beta\not=0$.
\end{rem}

We deal now with the general case where $(M,g)$ is complete with $\Ric\geq 0$.\\
We consider $p\in(1,\overline{p}_n)$, with a relative admissible pair of constants $(\alpha,\beta)$ as above in the compact case. It is easy to see (by continuity) that if $a\in(0,\alpha)$ is sufficiently close to $\alpha$, the pair $(a,\beta)$ is still admissible for $p$, then we define a slightly perturbed function $F$, 
\begin{align*}
F_a=&\,t(f_t-a  \vert\nabla f\vert^2-\beta e^{(p-1)f})\\
=&\,t(\Delta f+(1-a )\vert\nabla f\vert^2+(1-\beta)e^{(p-1)f})
\end{align*}
with $0<a<\alpha$. Repeating the previous computations~\eqref{eq960} in the compact case, we have
\begin{align}
(\partial_t-\Delta)F_a
\geq &\,F_a/t+2a  F_a^2/nt+2ta (1-\beta)^2e^{2f(p-1)}/n+[(p-1)-4a (1-\beta)/n]F_ae^{(p-1)f}\nonumber\\
&\,+2at(1-a )^2\vert\nabla f\vert^4/n-4a (1-a )\vert\nabla f\vert^2F_a/n\nonumber\\
&\,+2\nabla f\nabla F_a+t[(p-1)(\beta p-a )+4a (1-a )(1-\beta)/n]\vert \nabla f\vert^2 e^{(p-1)f}.\nonumber\\
\geq &\,F_a/t+2a  F_a^2/nt+2ta (1-\beta)^2e^{2f(p-1)}/n+[(p-1)-4a (1-\beta)/n]F_ae^{(p-1)f}\nonumber\\
&\,-4a (1-a )\vert\nabla f\vert^2F_a/n+2\nabla f\nabla F_a,\label{eq950}
\end{align}
since condition~\eqref{cond2} holds, with $a$ in place of $\alpha$.\\
We now want to compute the evolution equation of a ``localization'' of the function $F_a$, for $a >1$. We consider the following cut--off functions (of Li and Yau~\cite{liyau}): let $\psi:[0,+\infty)\to\R$ be a smooth function satisfying:
\begin{enumerate}
\item $\psi(s)=1$ if $s\in[0,1]$ and $\psi(s)=0$ if $s\in[2,+\infty)$,
\item $-C_1\leq \psi'(s)\leq 0$ for every $s\in[0,+\infty)$, that is $\psi$ is nonincreasing,
\item $|\psi''(s)|\leq C_2$ for every $s\in[0,+\infty)$,
\item $|\psi'(s)|^2/\psi(s)\leq C_3$ for every $s\in[0,+\infty)$ such that $\psi(s)\not=0$,
\end{enumerate}
for some positive constants $C_1,C_2,C_3$. Then, fixed $x_0\in M$ and $R>0$, denoting with $r=r(x)$ the distance function from the point $x_0$ in $(M,g)$, we define $\varphi(x)=\psi\bigl(\frac{r(x)}{R}\bigr)$.\\
We clearly have
$\nabla\varphi(x)=\psi'\bigl(\frac{r(x)}{R}\bigr)\frac{\nabla r(x)}{R}$, hence
\begin{equation}\label{est2}
\vert \nabla\varphi(x)\vert^2=\Bigl\vert\psi'\Bigl(\frac{r(x)}{R}\Bigr)\Bigr\vert^2\frac{1}{R^2}
\leq C_3\psi\Bigl(\frac{r(x)}{R}\Bigr)\frac{1}{R^2}=\frac{C_3\varphi(x)}{R^2}\leq \frac{C_3}{R^2}.
\end{equation}
when $\nabla\varphi$ exists (almost everywhere), being $\vert\nabla r(x)\vert^2=1$.\\
Thanks to the nonnegativity assumption on the  Ricci tensor, by the {\em Laplacian comparison  theorem} (see~\cite[Chapter~9, Section~3.3]{petersen2} and also~\cite{sheng1}), if $\Ric\geq0$, we have
\begin{equation}\label{est1}
\Delta r(x)\leq\frac{n-1}{r(x)}
\end{equation}
for every $x\in M$ {\em in the sense of support functions} (or {\em in the sense of viscosity}, see~\cite{crisli1} -- check also~\cite[Appendix~A]{manmasura} for comparison of the two notions), in particular, this inequality can be used in maximum principle arguments, see again~\cite[Chapter~9, Section~3]{petersen2}, for instance. Hence, we have $\Delta\varphi(x)=0$ if $r(x)\leq R$ and 
\begin{align}
\Delta \varphi(x)=&\,\psi''\Bigl(\frac{r(x)}{R}\Bigr)\frac{1}{R^2}+\psi'\Bigl(\frac{r(x)}{R}\Bigr)\frac{\Delta r(x)}{R}\\
\geq&\,-\frac{C_2}{R^2}-\frac{C_1(n-1)}{R\,r(x)}\\
\geq&\,-\frac{C_2}{R^2}-\frac{C_1(n-1)}{R^2}\\
=&\,-\frac{C_4}{R^2},
\end{align}
if $r(x)\geq R$. Thus, we have the general estimate on the whole $M$ (in the sense of support functions) $\Delta \varphi(x)\geq -{C_4}/{R^2}$.

Then, setting $Q (x,t) = \varphi(x) F_a (x,t)$, we compute (using formula~\eqref{eq950} and the inequalities above)
\begin{align*}
(\partial_t-\Delta)Q=&\,\varphi(\partial_t-\Delta)F_a -F_a \Delta\varphi-2
\nabla\varphi\nabla F_a \\
\geq&\,\frac{Q}{t}+\frac{2aQ^2}{nt\varphi}+\frac{2t\varphi a (1-\beta)^2}{n}e^{2f(p-1)}+Qe^{(p-1)f}\Bigl[(p-1)-\frac{4a (1-\beta)}{n}\Bigr]\\
&\,-\frac{4a(1-a)}{n}Q\vert\nabla f\vert^2+2\varphi\nabla f\nabla F_a-2\nabla\varphi\nabla F_a+F_a \frac{C_4}{R^2}.
\end{align*}
Arguing as before, setting $Q_{\min}(t)=\min_{x\in M}Q(x,t)$, we have $Q_{\min}(0)=0$ and 
$$
\varphi>0,\qquad 0=\nabla Q=\varphi\nabla F_a +F_a \nabla\varphi,\qquad \Delta Q\geq 0,
$$
at every point $x$ where the minimum of $Q$ is achieved. Hence (recalling that $a<\alpha\leq 1$), by Hamilton's trick, at almost every $t>0$ such that $Q_{\min}(t)<0$ we have
\begin{align*}
\frac{d\,}{dt}Q_{\min}
\geq&\,\frac{Q_{\min}}{t}+\frac{2aQ_{\min}^2}{nt\varphi}+\frac{2t\varphi a (1-\beta)^2}{n}e^{2f(p-1)}+Q_{\min}e^{(p-1)f}\Bigl[(p-1)-\frac{4a (1-\beta)}{n}\Bigr]\\
&\,+Q_{\min }\frac{C_4}{\varphi R^2}-\frac{4a(1-a)}{n}Q_{\min}\vert\nabla f\vert^2+2\varphi F_a \frac{\vert\nabla\varphi\vert^2}{\varphi^2}-2\varphi F_a \frac{\nabla f\nabla\varphi}{\varphi}\\
\geq&\,\frac{Q_{\min}}{t}+\frac{2aQ_{\min}^2}{nt\varphi}+\frac{2t\varphi a (1-\beta)^2}{n}e^{2f(p-1)}+Q_{\min}e^{(p-1)f}\Bigl[(p-1)-\frac{4a (1-\beta)}{n}\Bigr]\\
&\,+Q_{\min}\biggl[\frac{C_4}{\varphi R^2}+\frac{2C_3}{\varphi R^2}\biggr]-\frac{4a(1-a)}{n}Q_{\min}\vert\nabla f\vert^2+2Q_{\min}\frac{C_5\vert\nabla f\vert}{R\sqrt{\varphi}}\\
=&\,\frac{Q_{\min}}{t}+\frac{2aQ_{\min}^2}{nt\varphi}+\frac{2t\varphi a (1-\beta)^2}{n}e^{2f(p-1)}+Q_{\min}e^{(p-1)f}\Bigl[(p-1)-\frac{4a (1-\beta)}{n}\Bigr]\\
&\,+Q_{\min}\biggl[\frac{C_4}{\varphi R^2}+\frac{2C_3}{\varphi R^2}\biggr]-Q_{\min}\frac{4a(1-a )}{na ^2\varphi}\biggl[\vert\nabla f\vert^2\varphi-\frac{2C_6\vert\nabla f\vert\sqrt{\varphi}}{R}\biggl],
\end{align*}
where we used the fact that $\vert \nabla\varphi\vert\leq \sqrt{C_3\varphi}/{R}$, by inequality~\eqref{est2}, and we set $C_5=\sqrt{C_3}$, $C_6=\frac{C_5na ^2}{4a(1-a)}>0$.\\
Setting $y=\vert\nabla f\vert\sqrt{\varphi}\geq 0$, we see that the difference inside the last square brackets is given by $y^2-2C_6y/R$, which is then larger or equal than $-C_6^2/R^2$, hence
\begin{align*}
\frac{d\,}{dt} Q_{\min}
\geq&\,\frac{Q_{\min}}{t}+\frac{2aQ_{\min}^2}{nt\varphi}+\frac{2t\varphi a (1-\beta)^2}{n}e^{2f(p-1)}+Q_{\min}e^{(p-1)f}\Bigl[(p-1)-\frac{4a (1-\beta)}{n}\Bigr]\\
&\,+Q_{\min}\biggl[\frac{C_4}{\varphi R^2}+\frac{2C_3}{\varphi R^2}\biggr]+Q_{\min}\frac{4a(1-a)C_6^2}{na ^2\varphi R^2}\\
=&\,\frac{Q_{\min}}{t}+\frac{2aQ_{\min}^2}{nt\varphi}+\frac{2t\varphi a (1-\beta)^2}{n}e^{2f(p-1)}+Q_{\min}e^{(p-1)f}\Bigl[(p-1)-\frac{4a (1-\beta)}{n}\Bigr]\\
&\,+\frac{C_7Q_{\min}}{\varphi R^2}
\end{align*}
with $C_7=C_4+2C_3+4a(1-a)C_6^2/na^2$. Now we deal with the sum
\begin{align*}
&\frac{2t\varphi a (1-\beta)^2}{n}e^{2f(p-1)}+Q_{\min}e^{(p-1)f}\Bigl[(p-1)-\frac{4a (1-\beta)}{n}\Bigr]\\
&=\frac{2t\varphi a (1-\beta)^2}{n}w^2+Q_{\min}\Bigl[(p-1)-\frac{4a (1-\beta)}{n}\Bigr]w,
\end{align*}
where we set $w=e^{(p-1)f}\geq 0$, as in the compact case. The minimum of this quadratic in $w$ is given by
$$
-Q_{\min}^2\frac{[n(p-1)-4a(1-\beta)]^2}{8nta(1-\beta)^2\varphi}=
\frac{\varepsilon Q_{\min}^2}{t\varphi}-\frac{2aQ_{\min}^2}{nt\varphi}
$$
where $\varepsilon=\varepsilon(n,p,a ,\beta)$ is defined by equation~\eqref{epseq} and it is positive, being the pair $(a,\beta)$ admissible for $p$ (see condition~\eqref{cond1}).\\
Then, substituting in the previous formula and collecting similar terms, we finally get
\begin{align*}
\frac{d\,}{dt} Q_{\min}
\geq&\,\frac{Q_{\min}}{t}+\frac{\varepsilon Q_{\min}^2}{t\varphi}+\frac{C_7Q_{\min}}{\varphi R^2}.
\end{align*}
It follows, arguing as in the compact case that
$$
Q_{\min}\geq -\frac{\varphi}{\varepsilon}-\frac{tC_7}{\varepsilon R^2}\geq -\frac{1}{\varepsilon}-\frac{tC_7}{\varepsilon R^2}
$$
where we used the fact that $\varphi\leq 1$, hence
$$
\varphi(x)\frac{F_a (x,t)}{t}\geq \frac{Q_{\min}(t)}{t}\geq-\frac{1}{\varepsilon t}-\frac{C_7}{\varepsilon R^2}
$$
for every $x\in M$ and $t>0$. Sending $R\to+\infty$ (so that $\varphi \equiv 1$ on $M$), we then have
$$
\frac{F_a (x,t)}{t}\geq -\frac{1}{\varepsilon t}
$$
that is,
$$
f_t-a\vert\nabla f\vert^2-\beta e^{(p-1)f}\geq -\frac{1}{\varepsilon(n,p,a,\beta)t}
$$
and finally sending $a \to\alpha^-$, being $a\mapsto\varepsilon(n,p,a,\beta)$ a continuous function, we get the same conclusion~\eqref{eq970}
$$
f_t-\alpha \vert\nabla f\vert^2-\beta e^{(p-1)f}\geq-\frac{1}{\varepsilon(n,p,\alpha,\beta)t}
$$
that is,
\begin{equation}\label{eq990}
u_t-\alpha\frac{\vert\nabla u\vert^2}{u}-\beta u^p\geq -\frac{u}{\varepsilon(n,p,\alpha,\beta)t}
\end{equation}
for every $x\in M$ and $t>0$, given any pair $(\alpha,\beta)$ admissible for $p\in(1,\overline{p}_n)$.

Hence, we obtained the following proposition.

\begin{prop}\label{main}
Let $u:M\times[0,T)\to\R$ a classical positive solution of the equation $u_t = \Delta u + u^p$ with $p\in(1,\overline{p}_n)$, on an $n$--dimensional complete Riemannian manifold 
$(M,g)$ with nonnegative Ricci tensor. Then, for every pair $(\alpha,\beta)$ admissible for $p$, there holds
\begin{equation*}
u_t-\alpha\frac{\vert\nabla u\vert^2}{u}-\beta u^p\geq -\frac{u}{\varepsilon(n,p,\alpha,\beta)t}
\end{equation*}
for every $x\in M$ and $t>0$.
\end{prop}

\begin{rem}\label{remp1}
By carefully inspecting the proof, we can see that it works also if $p\in(0,1]$, moreover, we can choose $\alpha=\beta=1$ (check the last term of inequality~\eqref{eq960}), leading to the conclusion
\begin{equation*}
u_t-\frac{\vert\nabla u\vert^2}{u}-u^p\geq -\frac{2u}{nt}
\end{equation*}
for every $x\in M$ and $t>0$.
\end{rem}

\begin{rem}
Easily, considering a suitable multiple of the solution, a conclusion analogous to Proposition~\ref{main} also holds for the positive solutions of the equation $u_t = \Delta u + au^p$ with $a\in\R$ positive constant.
\end{rem}

\section{Harnack inequality}

We now derive the following Harnack inequality as a corollary of the semilinear Li~\&~Yau estimates.

\begin{prop}\label{harn}
Let $u:M\times[0,T)\to\R$ a classical positive solution of the equation $u_t = \Delta u + u^p$ with $p\in(1,\overline{p}_n)$, on an $n$--dimensional complete Riemannian manifold 
$(M,g)$ with nonnegative Ricci tensor. Then, for every pair $(\alpha,\beta)$ admissible for $p$ and $\varepsilon =\varepsilon(n,p,\alpha,\beta)$ as in formula~\eqref{epseq}, given any $0<t_1<t_2\leq T$ and $x_1,x_2 \in M$, the following inequality holds
\begin{equation*}
u(x_1,t_1) \leq u(x_2,t_2) \left(\frac{t_2}{t_1}\right)^{\!\!1/\varepsilon}\!\!\!\exp(\rho_{\alpha,\beta}(x_1,x_2,t_1, t_2)),
\end{equation*}
where 
\begin{equation*}
\rho_{\alpha,\beta} (x_1,x_2,t_1, t_2) = \inf_{\gamma \in \Gamma(x_1,x_2)} \int_{0}^{1} \left[\frac{|\dot{\gamma}(s)|^2}{4 \alpha (t_2-t_1)} - \beta (t_2-t_1) u^{p-1}(\gamma(s), (1-s)t_2+st_1)  \right]\, ds,
\end{equation*}
with $\Gamma(x_1,x_2)$ given by all the paths in $M$ parametrized by $[0,1]$ joining $x_2$ to $x_1$. In particular, since $u \geq 0$, we have the following bound
\begin{equation*}
u(x_1,t_1) \leq u(x_2,t_2) \left(\frac{t_2}{t_1}\right)^{\!\!1/\varepsilon}\!\!\!\exp \left(\inf_{\gamma \in \Gamma(x_1,x_2)} \int_{0}^{1} \frac{|\dot{\gamma}(s)|^2}{4 \alpha (t_2-t_1)} \, ds \right).
\end{equation*}
\end{prop}

\begin{proof}
We start with the inequality for $\log u$ that we derived in Section~~\ref{liyausem},
\begin{equation}\label{harn1}
(\log u)_t-\alpha \vert\nabla \log u \vert^2-\beta u^{p-1} \geq -\frac{1}{\varepsilon t}.
\end{equation}
Then, let $\gamma:[0.1]\to M$ be any curve such that $\gamma(0) = x_2$ and $\gamma(1) = x_1$ and consider the path $\eta:[0,1]\to M\times[t_1,t_2]$ defined as $\eta(s)= (\gamma(s), (1-s)t_2+st_1)$. Notice that $\eta(0)= (x_2, t_2)$ and $\eta(1)= (x_1, t_1)$. Integrating along $\eta$, thanks to estimate~\eqref{harn1}, we have
\begin{align*}
\log& \left(\frac{u(x_1,t_1)}{u(x_2,t_2)}\right) = \log u(x_1,t_1) - \log u(x_2,t_2)\\
&\, = \int_{0}^{1} \left( \frac{d\,}{ds} \log u (\eta(s)) \right) \, ds = \int_{0}^{1} \left[\langle \nabla \log u, \dot{\gamma} \rangle - (t_2-t_1) (\log u)_s \right] \, ds\\
&\, \leq \int_{0}^{1} \left[ \vert \nabla \log u \vert \vert \dot{\gamma}\vert + (t_2-t_1) \left(\frac{1}{\varepsilon [(1-s)t_2+st_1]} -\alpha \vert\nabla \log u \vert^2-\beta u^{p-1}\right) \right] \, ds\\
&\, = \int_{0}^{1} \frac{t_2-t_1}{\varepsilon [(1-s)t_2+st_1]} \, ds + \int_{0}^{1} \left[ \vert \nabla \log u \vert \vert \dot{\gamma}\vert - \alpha (t_2-t_1) \vert\nabla \log u \vert^2-\beta (t_2-t_1) u^{p-1}\right] \, ds\\
&\, \leq -\frac{\log[(1-s)t_2+st_1]}{\varepsilon}\,\biggr\vert_{0}^{1} + \int_{0}^{1} \left[\frac{\vert \dot{\gamma}\vert^2}{4\alpha (t_2-t_1)} -\beta (t_2-t_1) u^{p-1} \right] \, ds\\
&\, = \frac{1}{\varepsilon} \log \left( \frac{t_2}{t_1}\right) + \int_{0}^{1} \left[\frac{\vert \dot{\gamma}\vert^2}{4\alpha (t_2-t_1)} -\beta (t_2-t_1) u^{p-1} \right]\,ds
\end{align*}
Taking the exponential of this estimate, we obtain the desired inequality.
\end{proof}

\section{Ancient and eternal solutions}\label{anci}

Even if in the analysis of singularities of positive solutions of equation $u_t = \Delta u + u^p$, which we briefly discuss in Section~\ref{blowupsec}, only the eternal/ancient solutions in $\R^n$ are relevant, it is clearly of interest discussing (and possibly determining all) such solutions also on a general Riemannian manifold.
 
Let $u:M\times(-\infty,T)\to\R$ be a classical, positive, {\em ancient} solution of $u_t = \Delta u + u^p$, with $p\in(0,\overline{p}_n)$, where $(M,g)$ is a complete $n$--dimensional Riemannian manifold with nonnegative Ricci tensor.\\
By repeating the argument of the previous sections (or ``translating'' the solution) with $t-T_0$ in place of $t$, for any $T_0<T$, in the time interval $[T_0,T)$, we conclude that at every $x\in M$ and $t>T_0$ we have 
(estimate~\eqref{eq990})
$$
u_t\geq \alpha \frac{\vert\nabla u\vert^2}{u}+\beta u^p-\frac{u}{\varepsilon(t-T_0)}
$$
for every pair $(\alpha,\beta)$ admissible for $p$, where $\varepsilon>0$ is independent of $x$ and $t$.\\
Hence, sending $T_0\to-\infty$, we conclude
$$
u_t\geq \alpha \frac{\vert\nabla u\vert^2}{u}+\beta u^p>0
$$
for every $x\in M$ and $t\in(-\infty,T)$, if $\beta>0$ (which can always be chosen, by Remark~\ref{betanonzero}). Hence, we have the following consequence of Proposition~\ref{main}.

\begin{prop}\label{prop-mon}
Let $u:M\times(-\infty,T)\to\R$ a classical, positive, ancient solution of the equation $u_t = \Delta u + u^p$ with $p\in(0,\overline{p}_n)$, on an $n$--dimensional complete Riemannian manifold 
$(M,g)$ with nonnegative Ricci tensor. Then, there exist admissible pairs $\alpha,\beta\in(0,1)$ for $p$ and for any of such pairs there holds
\begin{equation}\label{mono}
u_t\geq\alpha\frac{\vert\nabla u\vert^2}{u}+\beta u^p\geq\beta u^p>0,
\end{equation}
for every $x\in M$ and $t\in(-\infty,T)$. In particular, for every fixed $x\in M$ the function $t\mapsto u(x,t)$ is monotone increasing.
\end{prop}

The following conclusion for eternal solutions is then straightforward, since any (non--identically zero) solution must blow up in finite time.

\begin{cor} In the same hypotheses of this proposition, $u$ is uniformly bounded above (in space and time) in every $M\times(-\infty,T')$, for any $T'<T$. Moreover, the function $t\mapsto\sup_{x\in M}u(x,t)$ is monotone increasing and
$$
\lim_{t\to-\infty}\sup_{x\in M}u(x,t)=0.
$$
If the function $u$ is an eternal nonnegative solution of $u_t = \Delta u + u^p$, with $p\in(0,\overline{p}_n)$, where $(M,g)$ is a complete $n$--dimensional Riemannian manifold with nonnegative Ricci tensor, it must necessarily be identically zero.
\end{cor}

For $p>1$, by direct integration of inequality~\eqref{mono} we get the estimate
$$
u(x,t)\leq \frac{C}{(\overline{T}-t)^{\frac{1}{p-1}}}
$$
for some $\overline{T}\geq T$ and $C=(\beta(p-1))^{-\frac{1}{p-1}}$. This "universal" decay estimate for ancient solutions on manifolds with nonnegative Ricci tensor was previously known, thanks to Corollary 3.7 in~\cite{cama2}, in the case of Riemannian manifolds with bounded geometry, but with a different proof based on a blow--up argument (see Section~\ref{blowupsec}).

By the same blow--up argument (see~\cite[Theorem~3.6 and Corollaries~3.7, 3.8]{cama2}) it can be shown that if in $\R^n$, for a fixed $p>1$, there are no nonzero, bounded, positive, eternal solutions of $u_t = \Delta u + u^p$, then the same holds for every complete $n$--dimensional Riemannian manifold with bounded geometry (notice that nonnegativity of the Ricci tensor is not necessary).\\
In the Euclidean space, it is well known that nontrivial global radial (static) solutions on $\mathbb{R}^n \times \mathbb{R}$ exist for any supercritical exponent $p \geq p_S=\frac{n+2}{n-2}$, see for instance Proposition~B at page~1155 in~\cite{guiniwang}. Conversely, while the triviality of eternal {\em radial} solutions can be shown in the full range of subcritical exponents $1<p<\frac{n+2}{n-2}$ (see Theorem~B at page~882 in~\cite{pqs2}), the same expected result for general (not necessarily radial) solutions was known only in the range $1<p< \frac{n(n+2)}{(n-1)^2}$ (see Theorem A page 882 in~\cite{pqs2}). Indeed, such triviality when $\frac{n(n+2)}{(n-1)^2} \leq p< p_S$ was a quite long standing open problem  (see~\cite{pqs1,pqs2}), which has been finally settled by Quittner for every $p>1$ when $n\leq 2$ in~\cite{quittner} and, very recently, for any $p< p_S$ when $n\geq 3$ in~\cite{quittner2}.\\
Then, by such optimal result of Quittner, in the general case of a Riemannian manifold, we can extend the triviality result for eternal solutions to Riemannian manifolds with bounded geometry (see Corollary~3.8 in~\cite{cama2}).

Unfortunately, our bound $\overline{p}_n$ on the exponent $p$ is actually smaller than $\frac{n(n+2)}{(n-1)^2}$, for every $n\geq 2$, but larger than $\frac{n}{n-2}$ (which appears in~\cite{quittner}) when $n\geq4$ (see the computation in the Appendix).\\

Turning our attention to the (only) ancient solutions of $u_t = \Delta u + u^p$ on a complete, $n$--dimensional Riemannian manifold with nonnegative Ricci tensor, we first notice that if $p\leq 1$, we have seen in Remark~\ref{remp1} that we can choose $\alpha=\beta=1$, hence
\begin{equation*}
u_t=\Delta u+u^p\geq\frac{\vert\nabla u\vert^2}{u}+u^p
\end{equation*}
that is,
\begin{equation*}
\Delta \log u\geq0.
\end{equation*}
If $M$ is compact, this clearly implies that every ancient, positive solution is trivial.\\
Not much is known, in a general Riemannian manifold, when $p>1$. By the work of Merle and Zaag~\cite{merlezaag} (and the above results of Quittner), it follows that in $\R^n$, if $1<p< p_S=\frac{n+2}{n-2}$ when $n\geq 3$, or for every $p>1$ if $n\leq 2$, any ancient, positive solution is trivial. We conjecture that this holds in general.

\begin{conge}
For any complete, $n$--dimensional Riemannian manifold $(M,g)$ with bounded geometry, if $p< p_S=\frac{n+2}{n-2}$ when $n\geq 3$, or for every $p>0$ if $n\leq 2$, any ancient, positive solution of $u_t = \Delta u + u^p$ is trivial. 
\end{conge}

The line of~\cite{merlezaag} is (at least apparently) difficult to be generalized to manifolds, possibly requiring in the computations/estimates the (quite restrictive) assumptions of parallel Ricci tensor and nonnegative sectional curvature. Anyway, for a compact Riemannian manifold with $\Ric>0$, we were able to obtain the triviality of ancient solutions in the same range of exponents $0\leq p<p_S$ (by the discussion above and Theorem~1.2 in~\cite{cama2}, in combination with the results of Quittner). In the next section we are trying to find out some other properties of ancient solutions.

\section{Convexity in time and monotonicity of the gradient}\label{monot}

The aim of this section is to discuss further qualitative properties of ancient, positive solutions of $u_t = \Delta u + u^p$, that can be derived analogously as the previous Li~\&~Yau type inequalities.

We now show that any ancient solution, apart from being monotone in time, as proved in Proposition~\ref{prop-mon}, is also convex in time, i.e. $\partial^2_t u > 0$, provided $n\geq 5$. Without loss of generality, we assume that $M$ is compact, since the noncompact case can be treated, as in the second part of Section~\ref{liyausem}, through a cut--off argument.

Let $u:M\times(-\infty,T)\to\R$ a classical, positive, ancient solution of the equation $$u_t = \Delta u + u^p$$ with $p\in(0,\overline{p}_n)$, on an $n$--dimensional, $n\geq 5$, complete Riemannian manifold 
$(M,g)$ with nonnegative Ricci tensor.
From Proposition~\ref{prop-mon}, for some admissible $0 <\alpha,\beta <1$, we proved
\begin{equation}\label{conv1}
u_t \geq \alpha \frac{|\nabla u|^2}{u}+ \beta  u^p.
\end{equation}
In particular, the function 
$$v=\partial_t u$$ is positive and satisfies
\begin{equation*}
\partial_t v - \Delta v = h v,
\end{equation*}
where we set $h=p u^{p-1}$. 
Let now $f= \log v$ and, observing that
\begin{equation*}
\partial_t f = \frac{v_t}{v} \quad 
\text{ and } \quad 
\Delta f = \frac{\Delta v}{v} - |\nabla f|^2,
\end{equation*}
we easily see that $f$ satisfies
\begin{equation*}
(\partial_t - \Delta) f = |\nabla f|^2 + h.
\end{equation*}
We introduce, for every $T_0<t<T$, the function
\begin{equation}\label{conv2}
F = (t-T_0)\left( \Delta f + (1-\alpha) |\nabla f|^2 + (1-\beta) h\right).
\end{equation}
A simple computation shows the following
\begin{align*}
(\partial_t - \Delta) \Delta f &= 2|\nabla^2 f|^2 + 2 Ric(\nabla f,\nabla f)+2  \langle \nabla \Delta f, \nabla f \rangle + \Delta h\\
(\partial_t - \Delta) f_i &= 2f_{ik} f_k + h_i\\
(\partial_t - \Delta) |\nabla f|^2 &= 4 f_{ik} f_i f_k + 2 h_i f_i - 2|\nabla^2 f|^2.
\end{align*}
From the above equations, using $\Ric\geq 0$, we deduce
\begin{align*}
(\partial_t - \Delta) F \geq \frac{F}{(t-T_0)} + (t-T_0) &\bigl[ 2|\nabla^2 f|^2 +   2  \langle \nabla \Delta f, \nabla f \rangle + \Delta h\\
&\,\,+ 4(1-\alpha) f_{ik} f_i f_k + 2(1-\alpha) h_i f_i\\
&\,\,- 2(1-\alpha) |\nabla^2 f|^2 + (1-\beta)(\partial_t - \Delta) h \bigr].
\end{align*} 
Notice that 
\begin{equation*}
2 \frac{\langle \nabla F, \nabla f \rangle}{(t-T_0)} = 2  \langle \nabla \Delta f, \nabla f \rangle + 4(1-\alpha) f_{ik} f_i f_k + 2(1-\beta) h_i f_i,
\end{equation*}
which can be plugged above to get
\begin{equation}\label{monot3}
(\partial_t - \Delta) F \geq \frac{F}{(t-T_0)} + 2 \langle \nabla F, \nabla f \rangle +  (t-T_0) \left[ 2 \alpha |\nabla^2 f|^2 + 2(\beta-\alpha) h_i f_i + \partial_t h -\beta(\partial_t - \Delta) h \right].
\end{equation} 
Now we see that
\begin{align*}
n |\nabla^2 f|^2\geq&\, (\Delta f)^2\\
=&\, \Bigl(\frac{F}{(t-T_0)} - (1-\alpha) |\nabla f|^2 - (1-\beta) h\Bigr)^2\\
=&\, \frac{F^2}{(t-T_0)^2} + (1-\alpha)^2 |\nabla f|^4 + (1-\beta)^2 h^2\\
&\,-\frac{2(1-\alpha)}{(t-T_0)} F |\nabla f|^2  -\frac{2(1-\beta)}{(t-T_0)} F h + 2 (1-\alpha)(1-\beta) |\nabla f|^2 h.
\end{align*}
Moreover, by direct computation,
\begin{equation*}
\partial_t h = p(p-1) u^{p-2} \partial_t u,
\end{equation*}
and 
\begin{equation*}
\Delta h = p(p-1) u^{p-2} \Delta u + p(p-1)(p-2) u^{p-3} |\nabla u|^2,
\end{equation*}
whence, using the equation for $u$ and inequality~\eqref{conv1}, we have
\begin{equation}
\begin{split}
\partial_t h - \beta (\partial_t -\Delta) h &= p(p-1) u^{p-2} \partial_t u - \beta p(p-1) u^{2(p-1)} + \beta p(p-1)(p-2) u^{p-3} |\nabla u|^2\\
&\geq \alpha p(p-1) u^{p-3} |\nabla u|^2 + \beta p(p-1)(p-2) u^{p-3} |\nabla u|^2\\
&= p(p-1) [\alpha + \beta(p-2)] u^{p-3} |\nabla u|^2\,.
\end{split}
\end{equation}
Hence, plugging everything in estimate~\eqref{monot3} we finally get
\begin{equation}\label{monot4}
\begin{split}
(\partial_t - \Delta) F \geq&\, \frac{F}{(t-T_0)} + 2 \langle \nabla F, \nabla f \rangle + \frac{2\alpha}{n} \frac{F^2}{(t-T_0)}\\
&\,+ \frac{2 \alpha}{n} (t-T_0) \Bigr[ (1-\alpha)^2 |\nabla f|^4 + (1-\beta)^2 h^2\\
&\,-\frac{2(1-\alpha)}{(t-T_0)} F |\nabla f|^2  -\frac{2(1-\beta)}{(t-T_0)} F h + 2 (1-\alpha)(1-\beta) |\nabla f|^2 h \Bigr]\\
&\,+ (t-T_0) \left[ 2(\beta-\alpha) h_i f_i + p(p-1) [\alpha + \beta(p-2)] u^{p-3} |\nabla u|^2 \right]\\
\geq&\, \frac{F}{(t-T_0)} + 2 \langle \nabla F, \nabla f \rangle + \frac{2\alpha}{n} \frac{F^2}{(t-T_0)}\\
&\,+ \frac{2 \alpha}{n} (t-T_0) \Bigl[ -\frac{2(1-\alpha)}{(t-T_0)} F |\nabla f|^2  -\frac{2(1-\beta)}{(t-T_0)} F h + 4 (1-\alpha)(1-\beta) |\nabla f|^2 h \Bigr]\\
&\,+ (t-T_0) \left[ 2(\beta-\alpha) h_i f_i + p(p-1) [\alpha + \beta(p-2)] u^{p-3} |\nabla u|^2 \right]\,.
\end{split}
\end{equation}
Let us first observe that $\alpha>\beta$ (see the Appendix). We claim that the quantity
\begin{equation*}
Q=-2(\alpha-\beta) h_i f_i + p(p-1) [\alpha + \beta(p-2)] u^{p-3} |\nabla u|^2 + \frac{8\alpha}{n} (1-\alpha)(1-\beta) |\nabla f|^2 h
\end{equation*}
is nonnegative. Notice that
\begin{align*}
2h_i f_i &= 2p(p-1) u^{p-2} u_i f_i = 2p(p-1) \frac{\sqrt{h}}{\sqrt{h}} u^{p-2} u_i f_i \\
&\leq p(p-1) \Bigl( \theta \frac{u^{2p-4}}{p u^{p-1}}|\nabla u|^{2}+\frac{1}{\theta} |\nabla f|^2 h \Bigr)\\
&= p(p-1) \Bigl( \frac{\theta}{p}u^{p-3}|\nabla u|^{2}+\frac{1}{\theta} |\nabla f|^2 h \Bigr)
\end{align*}
for all $\theta>0$. Thus we obtain
\begin{align*}
Q\geq &\,\left\{ p(p-1) [\alpha + \beta(p-2)]-\theta(p-1)(\alpha-\beta)\right\} u^{p-3} |\nabla u|^2 \\
&\,+\Bigr[\frac{8\alpha}{n} (1-\alpha)(1-\beta)-\frac{p(p-1)(\alpha-\beta)}{\theta} \Bigr]  |\nabla f|^2 h.
\end{align*}
Choosing
$$
\theta = \frac{np(p-1)(\alpha-\beta)}{8\alpha(1-\alpha)(1-\beta)}
$$
we get
\begin{equation}\label{eqapp}
Q\geq p(p-1)\Bigl[ \alpha + \beta(p-2)-\frac{n(p-1)(\alpha-\beta)^{2}}{8\alpha(1-\alpha)(1-\beta)}\Bigr]u^{p-3} |\nabla u|^2.
\end{equation}
Using (again) Mathematica\textsuperscript{TM} (see the Appendix), we have that the coefficient in the right hand side is nonnegative, hence $Q\geq 0$, if $n\geq 5$, $\alpha$ and $\beta$ are admissible and $p<\overline{p}_{n}$. Thus, at a minimum point of $F$ with $F\leq 0$, we have
$$
(\partial_t - \Delta) F \geq \frac{F}{(t-T_0)} + \frac{2\alpha}{n} \frac{F^2}{(t-T_0)}
$$
and, by maximum principle,
$$
F\geq -\frac{n}{2\alpha}
$$
which implies
$$
f_t \geq \alpha \vert \nabla f\vert^2+\beta h -\frac{n}{2\alpha (t-T_0)}\,.
$$
Then, recalling that $f= \log u_t$, we finally obtain
\begin{equation}\label{monot6}
\frac{\partial^2_t u}{u_t}\geq \alpha \frac{\vert\nabla u_t \vert^2}{u_t^2}+\beta h -\frac{n}{2\alpha (t-T_0)}.
\end{equation}
In particular, since $\alpha,\beta > 0$, we  have
\begin{equation}\label{monot7}
\frac{\partial^2_t u}{u_t}\geq -\frac{n}{2\alpha (t-T_0)}.
\end{equation}
Letting $T_0\to-\infty$, we conclude that at every $x\in M$ and $t<T$ we have $\partial^2_t u > 0$, for every pair $(\alpha,\beta)$ admissible for $p$. Therefore, we have showed the following proposition.

\begin{prop}\label{prop-conv}
Let $u:M\times(-\infty,T)\to\R$ a classical, positive, ancient solution of the equation $u_t = \Delta u + u^p$ with $p\in(0,\overline{p}_n)$, on an $n$--dimensional, $n\geq 5$, complete Riemannian manifold 
$(M,g)$ with nonnegative Ricci tensor. Then, there exist admissible pairs $\alpha,\beta\in(0,1)$ for $p$ and for any of such pairs there holds
$$
\partial^2_t u\geq \alpha \frac{\vert\nabla u_t \vert^2}{u_t}+\beta p u^{p-1}u_t >0
$$
for every $x\in M$ and $t\in(-\infty,T)$. In particular, for every fixed $x\in M$ the function $t\mapsto u(x,t)$ is convex.
\end{prop}

\begin{rem}
If the ambient is $\R^n$ and $u:\mathbb{R}^n\times(-\infty,T)\to\R$ a classical, positive, ancient solution of the equation
$$
u_t = \Delta u + u^p
$$
on $\mathbb{R}^n$, with $p\in(0,\overline{p}_n)$, letting $u_m = \partial_{x_m} u$, $m=1,..,n$, be any spatial derivative of $u$, we observe that 
\begin{equation*}
(\partial_t - \Delta) u_m = h u_m,
\end{equation*}
where $h= p u^{p-1}$, so that
\begin{equation*}
(\partial_t - \Delta) u_m^2 = 2 h u_m^2 - 2 |\nabla u_m|^2.
\end{equation*}
Thus, we consider, for $\eps>0$, the function $f_\eps = \frac12 \log(u_{m}^{2} + \eps)$. By direct computation we see that $f_\eps$ satisfies
\begin{equation*}
(\partial_t - \Delta) f_\eps = (1+\psi_\eps) h + |\nabla f_\eps|^2 (1+\phi_\eps),
\end{equation*}
whenever $u_m \ne 0$, with $\psi_\eps = -\frac{\eps}{u_m^2 + \eps}$ and $\phi_\eps = -\frac{\eps}{u_m^2}$. We introduce 
\begin{equation}\label{fave1}
F_\eps = t\left( \Delta f_\eps + (1-\alpha) |\nabla f_\eps|^2 + (1-\beta) h\right).
\end{equation}
In order to prove that $F_\eps \geq -C$ for some $C>0$, we notice that if $u_m = 0$, then $f_\eps$ achieves a (global) minimum point, so that $\Delta f_\eps \geq 0$, which in turn implies $F_\eps \geq 0$ by this equation. Hence, without loss of generality, we can suppose that $u_m^2 >0$, which implies that the perturbations $\psi_\eps$ and $\phi_\eps$, as well as all of their derivatives, tend to zero uniformly as $\eps \to 0$. In this way, if for $\eps = 0$ we let $f=f_0$ and $F=F_0$, we have that
\begin{equation*}
(\partial_t - \Delta) f =  h + |\nabla f|^2 + R(\eps),
\end{equation*}
with $R(\eps)\to 0$ as $\eps\to 0$. Observe now that $R(\eps)$ is smooth whenever $u_m \neq 0$ and that all of its derivatives tend to zero as $\eps\to 0$. Thus, we can omit the remainder $R(\eps)$ in the computations that we can perform as for the monotonicity of the time derivative $u_t$ proved above, only with the spatial derivative $\partial_{x_m} u$, $m=1,..,n$ in place of $u_t$. Following step by step the preceding computations, we then obtain the following fact.\\

\noindent{\em Let $u:\mathbb{R}^n\times(-\infty,T)\to\R$ a classical, positive, ancient solution of the equation $u_t = \Delta u + u^p$ with $p\in(0,\overline{p}_n)$, on $\mathbb{R}^n$, $n\geq 5$. Then, for every fixed $x\in \mathbb{R}^n$, the function $t\to|\nabla u|^2 (x,t)$ is monotone increasing.}\\

\noindent This could possibly be an initial step for an alternative line to show the Merle and Zaag result in~\cite{merlezaag} of triviality of ancient solutions in $\R^n$.
\end{rem}

\section{Singularities, blow--up and eternal/ancient solutions in $\R^n$}\label{blowupsec}

Let $u:M\times [0,T)\to\R$ a positive smooth solution of $u_t = \Delta u + u^p$ on an $n$--dimensional Riemannian manifold $(M,g)$, with $p>1$, such that $T<+\infty$ is the maximal time of existence. In this section we want to discuss, as in~\cite{pqs2}, the asymptotic behavior of a solution approaching the singular time $T$, by means of a slightly different blow--up technique, borrowed by the work of Hamilton about the Ricci flow~\cite{hamilton9} (and the mean curvature flow too).

Assuming that $M$ is compact and $u$ is uniformly bounded by some constant $A>0$, we have the evolution equations, 
\begin{align*}
\frac{d\,\,}{dt} u^2=&\,2u u_t=2 u\Delta u+2pu^{p+1}=\Delta u^2-2\vert \nabla u\vert^2+2pu^{p+1}
\leq\Delta u^2+2pA^{p+1}\\
\frac{d\,\,}{dt}u_t^2=&\,2u_t\Delta u_t +2pu^{p-1}u_t^2
=\Delta u_t^2-2\vert \nabla u_t\vert^2+2pu^{p-1}u_t^2\leq\Delta u_t^2+2pA^{p-1}u_t^2\\
\frac{d\,\,}{dt}\vert\nabla u\vert^2
=&\,2\nabla u\nabla\Delta u+2p\vert\nabla u\vert^2u^{p-1}\\
=&\,2\nabla u\Delta \nabla u-2\Ric(\nabla u, \nabla u)+2p\vert\nabla u\vert^2u^{p-1}\\
=&\,\Delta\vert\nabla u\vert^2-2\vert D^2u\vert^2-2\Ric(\nabla u, \nabla u)+2p\vert\nabla u\vert^2u^{p-1}\\
\leq&\,\Delta\vert\nabla u\vert^2+C\vert\nabla u\vert^2\\
\end{align*}
which, by means of maximum principle, imply that $u_t$  is also uniformly bounded (and $\vert\nabla u\vert^2$ too), then, for every $x\in M$, the limit
$$
u_T(x)=\lim_{t\to T}u(x,t)
$$
exists and it is Lipschitz. Repeating the same (standard) argument for $\vert\nabla^k u\vert^2$, for every $k\in\NN$, one can conclude that the limit map $u_T$ is smooth (and the convergence $u(\cdot,t)\to u_T$ also), hence we can ``restart'' the solution by standard methods (see~\cite{mantmart1}, for instance), getting a smooth solution in a larger time interval, in contradiction with the assumption that $T<+\infty$ was the maximal time of existence.

From this discussion, we conclude that 
$$
\limsup_{t\to T}u_{\max}(t)=\limsup_{t\to T}\max_{x\in M}u(x,t)=+\infty,
$$
moreover, such function $u_{\max}:[0,T)\to\R$, which is locally Lipschitz, by maximum principle must satisfy distributionally (see~\cite[Section~2.1]{Manlib}, for instance)
$$
u'_{\max}(t)\leq u_{\max}^p(t)
$$
implying (after integration of this differential inequality) the estimate
\begin{equation}\label{maxestim0}
u_{\max}(t)\geq\frac{1}{\bigl[(p-1)(T-t)\bigr]^{\frac{1}{p-1}}}=\overline{u}(t)
\end{equation}
which is the unique nonzero solution of the ODE $u'=u^p$, hence the unique nonzero trivial ancient solution of $u_t=\Delta u+u^p$, defined in the maximal interval $(-\infty,T)$.\\
Notice, that by the same argument, applied to $u_{\min}(t)=\min_{x\in M}u(x,t)$, satisfying $u'_{\min}(t)\geq u_{\min}^p(t)$, excludes the possibility that $T=+\infty$, if $u$ is not identically zero, that is, there are no positive immortal solutions, if $M$ is compact.

Then, we know that the following superior limit, in the case $M$ is compact,
\begin{equation}
\limsup_{t\to T}\max_{x\in M} u(x,t)^{p-1}({T-t})
\end{equation}
must be at least $1/(p-1)$. Up to our knowledge, we do not know if the same estimate holds for {\em complete--only} manifold $M$ (possibly with bounded geometry).\\
Assume that 
\begin{equation}\label{hyp}
\limsup_{t\to T}\max_{x\in M} u(x,t)^{p-1}({T-t})=+\infty\,.
\end{equation}
Then, let us choose a sequence of times $t_k\in[0,T-1/k]$ and points $x_k\in M$
such that
\begin{equation}\label{rateIIbis}
u(x_k,t_k)^{p-1}(T-1/k-t_k)=\max_{\genfrac{}{}{0pt}{}
{t\in[0,T-1/k]}{x\in M}}u(x,t)^{p-1}(T-1/k-t)\,.
\end{equation}
This maximum goes to $+\infty$ as $k\to\infty$, indeed, if it is
bounded by some constant $C$ on a subsequence $k_i\to\infty$, then for every $t\in[0,T)$ we have that definitely $t\in[0,T-1/k_i]$ and
\begin{equation*}
u(x,t)^{p-1}(T-t)=\lim_{i\to\infty}u(x,t)^{p-1}(T-1/k_i-t)\leq C
\end{equation*}
for every $x\in M$. This is in contradiction with the hypothesis~\eqref{hyp}.\\
This fact also forces the sequence $t_k$ to converge to $T$ as
$k\to\infty$. Indeed, if $t_{k_i}$ is a subsequence not converging to $T$, we
would have that the sequence $u(x_{k_i},t_{k_i})^{p-1}$ is
bounded, hence also
$$
\max_{\genfrac{}{}{0pt}{}
{t\in[0,T-1/{k_i}]}{x\in M}}u(x,t)^{p-1}(T-1/{k_i}-t)
$$
would be bounded.\\
Thus, we can choose an increasing (not relabeled) subsequence $t_k$
converging to $T$, such that $u(x_k,t_k)$ goes monotonically to
$+\infty$ and
$$
u(x_k,t_k)^{p-1}t_k\to+\infty\,,\qquad u(x_k,t_k)^{p-1}(T-1/k-t_k)\to+\infty\,,
$$
Moreover, we can also assume that $x_k\to \overline{x}$ for some $\overline{x}\in M$.\\
We rescale now {\em the Riemannian manifold $(M,g)$ and the solution $u$} as follows: for every $k\in\NN$, we consider on $M$ the rescaled metric $g_k=u^p(x_k,t_k)g$ and the function
$u_k: M\times I_k\to\Ri$, where
$$
I_k=\bigl[-u(x_k,t_k)^{p-1}t_k,u(x_k,t_k)^{p-1}(T-1/k-t_k)\bigr]\,,
$$
and $u_k$ is the function given by
$$
u_k(x,s)=\frac{u\bigl(x,s/u(x_k,t_k)^{p-1}+t_k\bigr)}{u(x_k,t_k)}\,.
$$
Then, there holds
\begin{align*}
\frac{\partial\,}{\partial s} u_k(x,s)=&\,\frac{u_t\bigl(x,s/u(x_k,t_k)^{p-1}+t_k\bigr)}{u^p(x_k,t_k)}\\
=&\,\frac{\Delta u\bigl(x,s/u(x_k,t_k)^{p-1}+t_k\bigr)}{u^p(x_k,t_k)}+
\frac{u^p\bigl(x,s/u(x_k,t_k)^{p-1}+t_k\bigr)}{u^p(x_k,t_k)}\\
=&\,\Delta_k u_k(x,s)+u_k^p(x,s)
\end{align*}
where $\Delta_k$ is the Laplacian associated to the Riemannian manifold $(M,g_k)$.\\
Hence, every $u_k$ is still a solution of our equation on $(M,g_k)$ and on a different time interval $I_k$. Moreover the following
properties hold,
\begin{itemize}
\item $u_k(x_k,0)=1$,
\item for every $\varepsilon>0$ and $\omega>0$ there exists $\overline{k}\in\NN$ such that 
\begin{equation}\label{ccc}
\max_{x\in M} u_k(x,s)\leq 1+\varepsilon
\end{equation}
for every $k\geq \overline{k}$ and $s\in[-u(x_k,t_k)^2t_k,\omega]$,
\end{itemize}
indeed (the first point is immediate), by the choice of the minimizing pairs 
$(x_k,t_k)$ we get
\begin{align*}
u_k(x,s)&\,=\frac{u(x,s/u(x_k,t_k)^{p-1}+t_k)}{u(x_k,t_k)}\\
&\,\leq \frac{u(x_k,t_k)}{u(x_k,t_k)}\,\frac{T-1/k-t_k}{T-1/k-t_k-s/u(x_k,t_k)^{p-1}}\\
&\,= \frac{T-1/k-t_k}{(T-1/k-t_k)-s}\,,
\end{align*}
if
$$
\frac{s}{u(x_k,t_t)^{p-1}}+t_k\in[0,T-1/k],
$$
that is, if $s\in I_k$. Then, assuming $s\leq\omega$ and $k$ large enough,  
the claim follows as we know that  $u(x_k,t_k)^{p-1}(T-1/k-t_k)\to+\infty$.

If we now take a (subsequential) limit of these {\em pairs} manifold--solution, clearly the manifolds $(M,g_k)$ under the {\em pointed convergence} in $x_k$ (see~\cite{petersen2}, for instance) converge to $\R^n$ with its flat metric and the solutions $u_k$ converge smoothly in every compact time interval of $\R$ (by standard local uniform parabolic estimates -- similar to the ones at the beginning of the section, for instance) to a smooth eternal solution $u_\infty:\R^n\times\R\to\R$, bounded with all its derivatives of the same semilinear equation $u_t=\Delta u+u^p$, notice indeed that the time interval of existence is
the whole $\R$, as $\lim_{k\to\infty} I_k=(-\infty,+\infty)$. Moreover, the function $u_\infty$ takes its absolute maximum, which is 1, at time $s=0$ at the origin of $\R^n$, hence the limit flow is nonzero.

Since we know that for $p<p_S=\frac{n+2}{n-2}$, by the works of Quittner~\cite{quittner,quittner2}, such eternal solutions in $\R^n$ do not exist, we can conclude that, for such range of exponents, 
\begin{equation}\label{hyp2}
\limsup_{t\to T}\max_{x\in M}u(x,t)^{p-1}({T-t})=C<+\infty\,.
\end{equation}
Hence, there exists a constant $\overline{C}$ such that
$$
u(x,t)\leq\frac{\overline{C}}{(T-t)^{\frac{1}{p-1}}}
$$
for every $x\in M$ and $t\in[0,T)$.\\
In such situations, let us choose a sequence of times $t_k\nearrow$ and points $x_k\in M$ such that
\begin{equation}
u(x_k,t_k)=\max_{x\in M}u(x,t_k)\qquad\qquad\text{ and }\qquad\qquad u(x_k,t_k)^{p-1}({T-t_k})\to C
\end{equation}
(we can also assume that $x_k\to \overline{x}$, for some $\overline{x}\in M$). Repeating the above blow--up procedure, that is,
considering on $M$ the rescaled metric $g_k=u^p(x_k,t_k)g$ and the function
$$
u_k(x,s)=\frac{u\bigl(x,s/u(x_k,t_k)^{p-1}+t_k\bigr)}{u(x_k,t_k)},
$$
where
$$
I_k=\bigl[-u(x_k,t_k)^{p-1}t_k,u(x_k,t_k)^{p-1}(T-t_k)\bigr)\,,
$$
we get this time as a limit an ancient smooth solution $u_\infty:\R^n\times\R\to\R$ of $u_t=\Delta u+u^p$, defined in the time interval $(-\infty,C)=\lim_{k\to\infty} I_k$. As $u_k(x_k,0)=\max_{x\in M}u_k(x,0)=1$, we have $u_\infty(0,0)=\max_{x\in\R^n}u_\infty(x,0)=1$ and moreover, since by estimate~\eqref{maxestim0} we have
$$
\max_{x\in M}u_k(x,s)^{p-1}\geq\frac{1}{p-1}\,\frac{1}{u(x_k,t_k)^{p-1}(T-t_k) - s}
$$
passing to the limit, we conclude
$$
\max_{x\in\R^n}u_\infty(x,s)^{p-1}\geq\frac{1}{p-1}\frac{1}{C-s},
$$
since
\begin{equation}
u(x_k,t_k)^{p-1}({T-t_k})\to C\,.
\end{equation}
This clearly shows that the interval $(-\infty,C)$ is maximal for $u_\infty$ and $u_\infty$ is nonzero.\\
If now we consider any $\varepsilon>0$, there exists $\overline t\in[0,T)$ such that
\begin{equation}
u(x,t)^{p-1}({T-t})\leq C+\varepsilon
\end{equation}
for every $t\in(\overline{t},T)$ and $x\in M$, hence
\begin{align*}
u_k(x,s)^{p-1}&\,=\frac{u(x,s/u(x_k,t_k)^{p-1}+t_k)^{p-1}}{u(x_k,t_k)^{p-1}}\\
&\,\leq \frac{C+\varepsilon}{u(x_k,t_k)^{p-1}(T-s/u(x_k,t_k)^{p-1}-t_k)}\\
&\,= \frac{C+\varepsilon}{u(x_k,t_k)^{p-1}(T-t_k)-s}\,,
\end{align*}
if $s/u(x_k,t_t)^{p-1}+t_k\in(\overline{t},T)$, that is, if 
$$
s\in\bigl( (\overline{t}-t_k)u(x_k,t_t)^{p-1},(T-t_k)u(x_k,t_t)^{p-1}\bigr)\,.
$$
Passing to the limit, we conclude that for every $s\in(-\infty,C)$ and $x\in\R^n$, there holds
$$
u_\infty(x,s)^{p-1}\leq \frac{C+\varepsilon}{C-s},
$$
hence, by the arbitrariness of $\varepsilon$
$$
u_\infty(x,s)^{p-1}\leq \frac{C}{C-s},
$$
for every $x\in\R^n$ and $s\in(-\infty,C)$. Notice that equality holds at $x=0$ and $s=0$.\\
Then, we have
$$
\frac{1}{p-1}\frac{1}{C-s}\leq\max_{x\in M}u_\infty(x,s)^{p-1}\leq\frac{C}{C-s}
$$
where the second inequality is an equality at $s=0$.\\
By means of the result of Merle and Zaag~\cite{merlezaag}, this solution $u_\infty$ is actually trivial (that is, constant in space), hence 
$$
u_\infty(x,s)=\frac{1}{\bigl[(p-1)(C-s)\bigr]^{\frac{1}{p-1}}}\,.
$$
Being $u_\infty(0,0)=1$, it follows that $C=\frac{1}{p-1}$, hence
$$
u_\infty(x,s)=\frac{1}{\bigl[1-(p-1)s\bigr]^{\frac{1}{p-1}}},
$$
defined on $(-\infty,1/(p-1))$, moreover
\begin{equation*}
\limsup_{t\to T}\max_{x\in M}u(x,t)^{p-1}({T-t})=\frac{1}{p-1}\,.
\end{equation*}

This clearly gives a ``universal'' asymptotic profile at a singularity of any solution.

\newpage

\appendix
\section*{Appendix}

With the next computation with Mathematica\textsuperscript{TM}, we find the maximal exponent $\overline{p}_n>1$, given in Definition~\ref{pienne} such that there exists admissible constants $\alpha$ and $\beta$. Moreover, we also check that $\overline{p}_n$ is smaller than $\frac{n(n+2)}{(n-1)^2}$ and larger than $\frac{n}{n-2}$, for every $n\geq4$.

\ 

\hspace{-3em}{\includegraphics[scale=0.62]{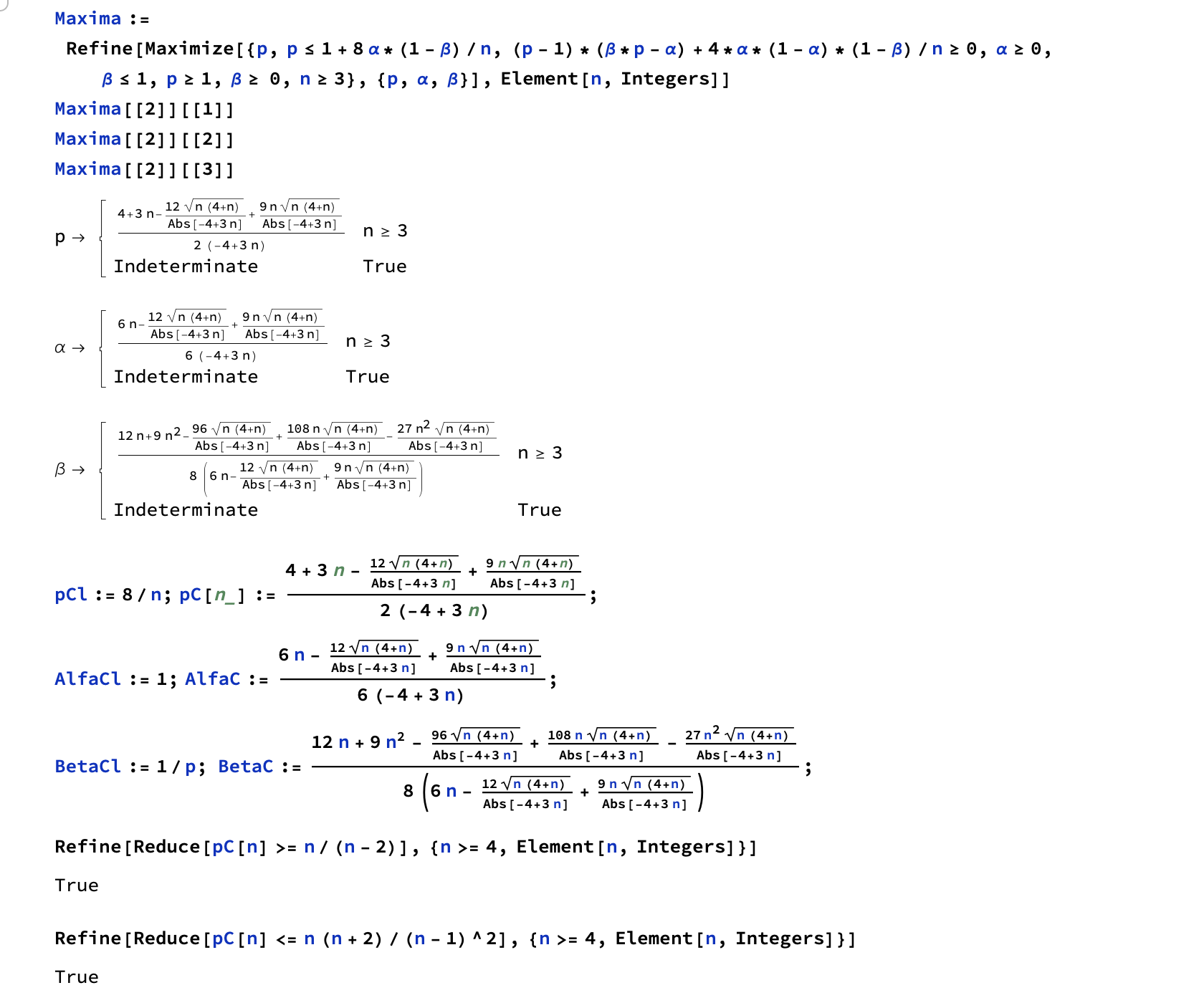}}

\eject

Here instead, we show that the previously found constants satisfy $\alpha>\beta$ and that the coefficient in the right hand side of equation~\eqref{eqapp} is nonnegative.

\ 

\hspace{-1.5em}{\includegraphics[scale=0.56]{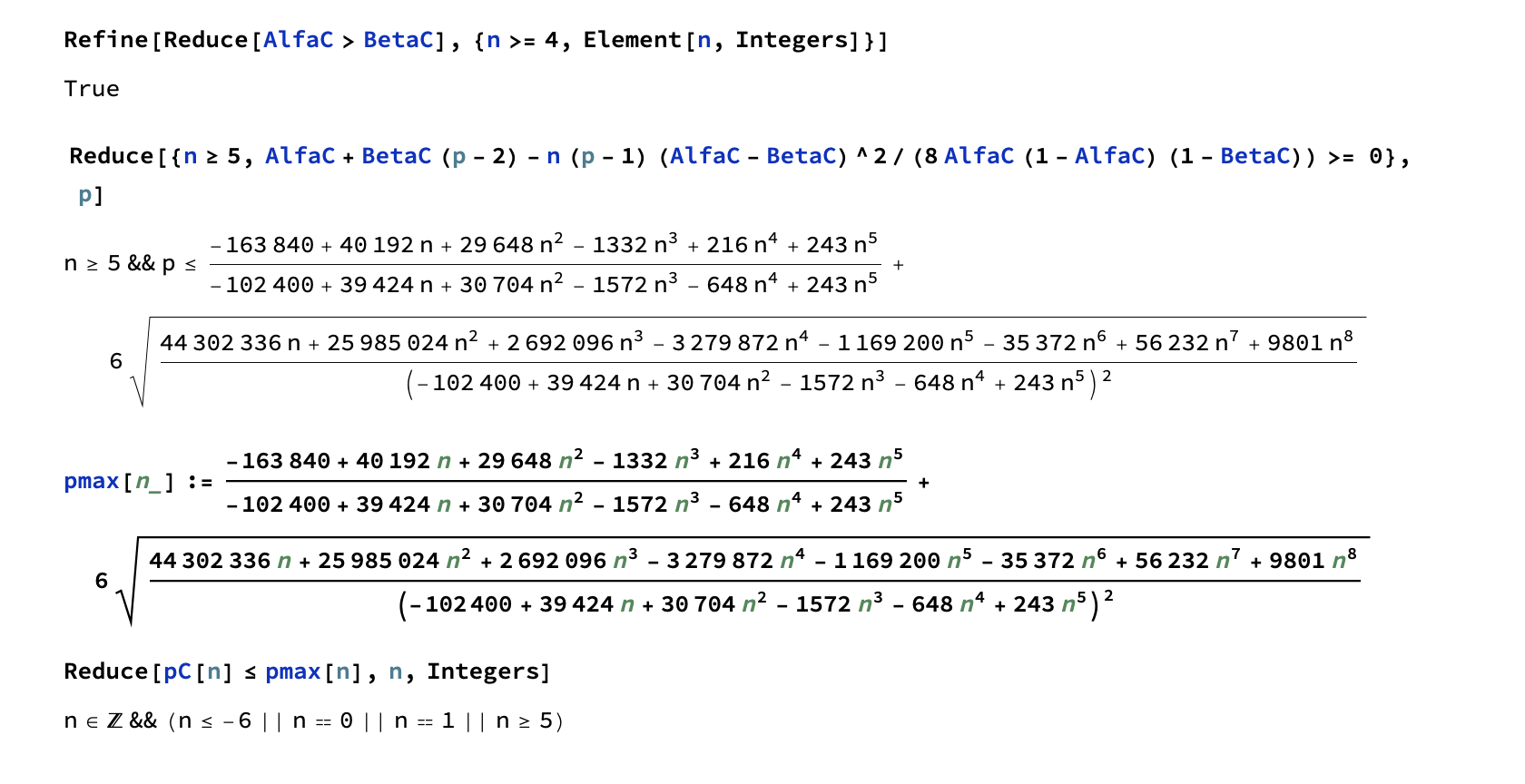}}

\

\bibliographystyle{amsplain}
\bibliography{biblio}

\end{document}